\documentclass[11pt,twoside]{article}

\usepackage{mathrsfs,amsfonts,amsmath,amssymb}
\usepackage{latexsym}
\newtheorem{satz}{Theorem}

\newtheorem{theorem}[satz]{Theorem}
\newtheorem{lemma}[satz]{Lemma}

\newtheorem{corollary}[satz]{Corollary}

\def\no{\noindent}
\def\sbeq{\subseteq}

\def\Z{\mathbb {Z}}

\def\e{\varepsilon}

\def\s{\sigma}

\def\S{\mathcal{S}}

\def\({\big (}
\def\){\big )}

\def\G{\mathcal{G}}
\def\X{\mathcal{X}}

\def\le{\leqslant}
\def\ge{\geqslant}
\def\_phi{\varphi}

\def\snd{{\rm {snd}}}

\def\B{\Big}

\begin{document}

\title{\bf An upper bound for  weak $B_k$--sets }

%\author{ By\\  \\{\sc Tomasz Schoen\footnote{\textsc{The author is supported by NCN grant .....} }}}
\author{ By\\  \\{\sc Tomasz Schoen\footnote{The author is supported by NSC grant ...}
~ and ~ Ilya D.~Shkredov
%\footnote{This work was supported by grant RFFI NN
%06-01-00383, 11-01-00759 and grant Leading ScientificSchools N 8684.2010.1.}
} }
\date{}

\maketitle

%\begin{dedication}
%\end{dedication}

\begin{abstract} We prove that if $A\sbeq [N]$ does not contain any solution to the equation
$x_1+\dots+x_k=y_1+\dots+y_k$ with distinct $x_1,\dots,x_k,y_1,\dots,y_k\in A$, then $|A|\ll {k^{3/2}}N^{1/k}.$
\end{abstract}

%\centerline{\bf Introduction}
\bigskip

\section{Introduction}
A set $A\sbeq \Z$ is called a $B_k$--set, $k\ge 2,$ if
\begin{equation}\label{B_k}
x_1+\dots+x_k=y_{1}+\dots+y_{k}
\end{equation}
with $x_i,y_i\in A$ implies that $x_1,\dots,x_k$ is a permutation of $y_{1},\dots,y_{k}.$ If $A$ is a $B_k$-subsets of $[N]$, then even a  simple argument argument  $\binom{|A|+k-1}{k}=|kA|\le kN$ shows that
$$|A|\le (kk!)^{1/k} N^{1/k}.$$
Currently best bound
 $$|A|\le (1+o(1))(1+\e_k)\frac{k}{2e}N^{1/k}$$
 was proved by Green \cite{green} (see also \cite{obryant} page 14), where $\e_k\rightarrow 0$ as $k\rightarrow \infty$.
 The history of the problem and bibliography can be found in \cite{obryant}.

We call a set $A\sbeq \Z$ a weak $B_k$--set ($B^*_k$--set) if
every solution to (\ref{B_k})
with $x_i\in A$ implies that among numbers $x_1,\dots,x_k,y_{1},\dots,y_{k}$ at least two are equal. In other words $A$ does not contain a solution to (\ref{B_k}) with distinct elements.
It is already non-trivial to prove that  every $B^*_k$--subset of $[N]$ has size $O_k(N^{1/k}).$ It is a fact that there are equations (see Theorem 3.3 in \cite{ruzsa-I}) such that the threshold for the size of a set without solutions in distinct integers is of different order of magnitude.  Ruzsa \cite{ruzsa-I} proved that for every  $B^*_k$--set  $A\sbeq [N]$ we have
$$|A|\le (1+o(1))k^{2-1/k}N^{1/k} \,.$$
Very recently Timmons \cite{timmons} showed that
$$|A|\le (1+o(1))(1+\e_k)\frac{k^{2}}{4}N^{1/k},$$
 where $\e_k\rightarrow 0$ as $k\rightarrow \infty$. On the other hand, for $k\ge 3$ Timmons constructed  a $B^*_k$--subset of $[N]$ with $(1-o(1))2^{1-1/k}N^{1/k}$ elements. Our aim is to prove  the following theorem.

\begin{theorem}
    Let $A\subseteq [N]$ be a  $B^*_k$--set and $N\ge (2 k^{9/2})^k$.
    Then
$$
    |A| \le 16  k^{3/2} N^{1/k} \,.
$$
\label{t:main_introduction}
\end{theorem}

\section{The proof of the main result}

For a finite subset of integers $A$ put
$$\s_k(n) = \s_k(n;A)= |\{(x_1,\dots,x_k)\in A^k: x_1+\dots+x_ k=n\} | \,,$$
$$f(t)=f_A(t)=\sum_{a\in A}e^{-2\pi i at} \,, \quad \quad t\in [0,1]$$
further
$$M_k = M_k (A) =\sum_n \s_k(n)^2=\int_0^1|f(t)|^{2k}dt \,,$$
$$s_k(n) = s_k (n;A) = | \{(x_1,\dots,x_k)\in A^k,\, x_i\not=x_j: x_1+\dots+x_ k=n\} |,$$
$$S_k = S_k (A) =\sum_n s_k (n;A)^2 = \sum_n s_k (n)^2
    =
    $$
\begin{equation}\label{f:S_k_equality}
    =
      | \{ (x_1,\dots,x_k, y_1,\dots,y_k )\in A^{2k}~:~ x_1+\dots+x_ k =  y_1 +\dots+y_k
         ,\, x_i\not=x_j,\, y_i \not=y_j\}|  \,.
\end{equation}
Furthermore, define
\begin{eqnarray*}
S^*_k(A)= | \{ (x_1,\dots,x_k, y_1,\dots,y_k )\in A^{2k}&:&x_1+\dots+x_ k =  y_{1} +\dots+y_{k}
         ,\\ && x_i \not=x_j,\, y_i\not=y_j,\, x_i\not=y_j,\, 1\le i,j\le k \}|  \,.
         \end{eqnarray*}
and for finite subsets of integers $A_1,A_2$ put
\begin{eqnarray*}
S^*_k(A_1,A_2)= | \{ (x_1,\dots,x_k, y_1,\dots,y_k )\in A_1^{k}\times A_2^k &:&x_1+\dots+x_ k =  y_{1} +\dots+y_{k}
         ,\\ && x_i \not=x_j,\, y_i\not=y_j,\, x_i\not=y_j,\, 1\le i,j\le k \}|  \,.
         \end{eqnarray*}
Clearly, $A$ is a  $B^*_k$--set $A$ if and only if $S^*_k (A)=0$

\medskip
Our first lemma can be extracted from Ruzsa's paper \cite{ruzsa-I}.
We will rewrite it for a convenient for us form.

\begin{lemma}\label{l_n:ruzsa}
    Let $A$ be a finite set of integers with $|A| \ge 16 k^4$.
    Then
\begin{equation}\label{e:ruzsa-S_k}
    2^{-k} k!|A|^k \le S_k \le M_k \le 2 S_k \,.
\end{equation}
\label{l:M_k=S_k}
\end{lemma}
\begin{proof}
    If $k=1$, then the result is trivial, thus suppose that $k\ge 2$.
First observed that
\begin{equation}\label{inequality-Mk}
    M_k\ge S_k\ge k!|A|(|A|-1)\dots (|A|-k+1) \ge (5/4)^{-k} k! |A|^k \ge 2^{-k} k!|A|^k
    \,.
\end{equation}
    Ruzsa proved (see inequality (5.8) in \cite{ruzsa-I}) that
$$
    \sum_n (\sigma_k (n) - s_k (n))^2 \le k^4 M_{k-1} \,,
$$
    hence by the H\"{o}lder inequality
\begin{equation*}\label{tmp:03.06.2016_1}
    \sum_n (\sigma_k (n) - s_k (n))^2 \le k^4 M^{\frac{k-2}{k-1}}_k |A|^{\frac{1}{k-1}} \,.
\end{equation*}
Therefore, by the triangle inequality
$$M^{1/2}_k\le S^{1/2}_k+k^2 M^{\frac{k-2}{2(k-1)}}_k |A|^{\frac{1}{2(k-1)}}\,.$$
By the middle inequality in (\ref{inequality-Mk}) and the assumption $|A| \ge 16 k^4$ we see that
$$
    |M^{1/2}_k - S^{1/2}_k|^2 \le k^4 M^{\frac{k-2}{k-1}}_k |A|^{\frac{1}{k-1}}
    %\ll \Delta
    \le 2^{-4} M_k
    %\,,
$$
so that
$$
    M_k \le 2 S_k
$$  as required.
\hfill$\blacksquare$
\end{proof}

\bigskip

Observe  that, actually, the inequality (\ref{e:ruzsa-S_k}) can be strengthened to $M_k = S_k (1+o(1))$ for large $|A|$.
Our next lemma provides a  straightforward relation  between $S_k$ and $S_l$ for $l<k$.

\begin{lemma}
    Let $1\le l<k$ and let $A$ be a $B^*_k$--set with $|A| \ge 16 k^4$.
    Then
$$
    S_k \le (\sqrt{2} k)^{2{(k-l)}} |A|^{k-l} S_{l} \,.
$$
\label{l:S_k,S_l}
\end{lemma}
\begin{proof}
    Indeed, counting the number of the solutions to (\ref{f:S_k_equality}), we see that by the assumption on  $A$  we must have $x_i=y_j$ for some $i$ and $j$. There are at most $k^2$ choices for $i$ and $j$ and thus
\begin{equation}\label{ruzsa-k2_my}
    S_k \le k^2 |A| S_{k-1} \le k^2 |A| M_{k-1} \,.
\end{equation}
    Furthermore, by the H\"older inequality and Lemma \ref{l:M_k=S_k}
$$
    M_{k-1}
        =
            \int_0^1 |f(t)|^{2k-2}dt
                \le \B(\int_0^1 |f(t)|^{2k}dt\B)^{\frac{k-l-1}{k-l}}
                        \B(\int_0^1 |f(t)|^{2l}dt\B)^{\frac1{k-l}}
                            \le
                                (2S_k)^{\frac{k-l-1}{k-l}}(2S_l)^{\frac1{k-l}} \,,
$$
hence by (\ref{ruzsa-k2_my})
$$ S_k\le 2^{}k^2|A| S_k^{\frac{k-l-1}{k-l}} S_l^{\frac1{k-l}}$$
and the assertion follows.
%    as required.
\hfill$\blacksquare$
\end{proof}

\bigskip

Using a simple probabilistic argument we prove that the variables  from both sides of (\ref{f:S_k_equality}) can be chosen from disjoint sets.
This allows us to apply a version of a well known theorem of Erd{\H o}s--Ko--Rado, see Theorem \ref{t:EKR} below.

\begin{lemma}\label{l:S_k=S_k(A,A^c)}
 Let $A$ be finite set of integers. Then there exists a set  $A' \subseteq A$ such that
$$
    S^*_k (A) \le
    4^k  S^*_k (A',A\setminus A')\,.
$$\end{lemma}
\begin{proof} Let us pick, randomly, independently, each element of $A$ with probability $1/2.$
    Let  $A'$ be the  random set of chosen elements. Clearly,
   % $$
   %    \mathbb{E} |A'|  = |A| /2
   %         %\,.
   % $$
   % and from the large deviations inequality (see e.g. \cite{TV}, Theorem 1.8) we have  $|A|/4 \le |A'| \le 3|A|/4$ with probability tending to $1$, %as $|A|\rightarrow \infty$.
   % %Clearly,
   % Furthermore,
\begin{equation*}\label{tmp:04.06.2016_1}
    \mathbb{E} S^*_k (A',A\setminus A')
        = \sum_{x_1+\dots+x_k = x'_1 + \dots +x'_k }2^{-2k}  \,,
\end{equation*}
where the summation is taken over all solutions in the set $A$ with distinct integers.
Whence,
$$\mathbb{E} S^*_k (A',A\setminus A')=4^{-k}S^*_k(A)\,,$$
so that
$$\mathbb{P}\big(\,S^*_k (A',A\setminus A')\ge 4^{-k}S^*_k(A)\, \big)>0\,.$$
%and by the Markov inequality the probability that $S^*_k (A',A\setminus A')\ge 2\cdot 4^{-k}S^*_k(A)$ is less or equal $1/2.$
Therefore, there is a specific choice of $A'$ such that  $S^*_k (A) \le
    4^k  S^*_k (A',A\setminus A').$
 %   Because of the sum in  (\ref{tmp:04.06.2016_1}) is taken over distinct $x_1,\dots,x_k$ as well as over distinct $x'_1,\dots,x'_k$ we see that all %random variables in the formula are independent.
 %   Hence
%$$
%    \mathbb{E} S^*_k (A',A\setminus A') = \sum^*_{x_1+\dots+x_k = x'_1 + \dots +x'_k }
%            \mathbb{E} \xi_{x_1} (\omega) \dots \mathbb{E} \xi_{x_k} (\omega)
%                \mathbb{E} (1-\xi_{x'_1} (\omega)) \dots \mathbb{E} (1-\xi_{x'_k} (\omega))
%$$
%$$
%                    =
%                    2^{-2k} S^*_k (A)
%                    %\,.
%$
\hfill$\blacksquare$
\end{proof}

\bigskip

For a   $B_k^*$--set  $A$  define
$$
    \mathcal{G}  = \mathcal{G} (A) = \{ j \in [k] ~:~ A \text{ is a $B_j^*$--set} \}
        \quad \quad \mbox{ and } \quad \quad
            \mathcal{X} =  \mathcal{X} (A) = [k] \setminus \mathcal{G} (A) \,.
$$

Timmons \cite{timmons} proved that a large subset of $A$ is a  $B^*_{\lfloor k/2\rfloor}$  or $B^*_{\lceil k/2\rceil}$--set.
We prove that there is a large set $B\sbeq A$ such that  the set $\mathcal{G}(B)$ is very large and  structural. This is a crucial argument in our approach.
%
%Now we need a lemma on the structure of a set $\mathcal{G} (A_1)$ for large subset of $A$.

\bigskip

\begin{lemma}
    Let $A$ be a $B_k^*$--set.
    Then there is a set $B \subseteq A$,
    $$
        |B| \ge |A| - 2 k^3
    $$
    such that for $\mathcal{X} = \mathcal{X} (B)$, $\G = \G(B)$ we have for any $l$ we have $l \mathcal{X} \cap \G = \emptyset \,.$
    %one has
    In particular,
\begin{equation}\label{in:X}
    |\X\cap [l]|\le  \big\lfloor\frac{l}2\big\rfloor
\end{equation}
\label{l:structure_X}
  for every $1\le l\le k.$
\end{lemma}
\begin{proof}
We prove the lemma using an iterative procedure. We start with the sets $A_0=A$, $\G_0=\emptyset,$ $\X_0=\emptyset$ and $\S_0=\emptyset.$ Suppose that we have chosen
$A_i, \G_i, \X_i$ and $\S_i.$ If $[k]=\G_i\cup \X_i$ we stop the process, otherwise put  $l:=\min([k]\setminus (\G_i\cup \X_i))$. If $A_i$ contains $k$ pairwise disjoint solutions to the equation
\begin{equation}\label{eq-l}
x_1+\dots+x_{l}=y_1+\dots+y_{l}
\end{equation}
with all variables distinct, and disjoint with $\S_i,$  then we put
\begin{equation}\label{step-1}
A_{i+1}=A_i,\,~ \G_{i+1}=\G_i,\,~ \X_{i+1}=\X_i\cup \{l\},\,~\S_{i+1}=\S_{i}\cup \textsf{Sol}_i\,,
\end{equation}
where $\textsf{Sol}_i$ is the set of all numbers involved in the $k$ disjoint solutions to (\ref{eq-l}).
If this is not the case (there are less than $k$ such solutions), let $\textsf{Sol}_i$ be a maximal collection of pairwise disjoint solutions to (\ref{eq-l}) with distinct variables,
and disjoint with $\S_i$. Then
we put
\begin{equation}\label{step-2}
A_{i+1}=A_i\setminus ( \S_i\cup \textsf{Sol}_i),\,~ \G_{i+1}=\G_i\cup\{l\},\,~ \X_{i+1}=\emptyset,\,~ \S_{i+1}=\emptyset\,.
\end{equation}
Observe that then $A_{i+1}$ is a $B^*_l$--set, and $l\in \G_j,$ for all $ j\ge i+1.$  Thus,
the step (\ref{step-2}) can be applied at most $k$ times, as in each application we add one element to the set $\G_i.$ Therefore, the process must terminate after  $m\le k^2$ iterations. We put $B=A_m, \G=\G_m$ and $\X=\X_m.$ Since
$|\S_i\cup \textsf{Sol}_i|\le 2k^2$
it follows that
$$|B|\ge |A|-2k^3.$$

Next, we prove that $l \mathcal{X} \cap \G = \emptyset \,.$ Suppose that it does not hold. Then there exist $n_1,\dots,n_l\in \X,\, l\le k,$ such that $n_1+\dots+n_l=n\in \G.$
By our construction there are pairwise disjoint solutions in distinct integers
\begin{eqnarray*}
x^{(1)}_1+\dots+x^{(1)}_{n_1}&=&y^{(1)}_1+\dots+y^{(1)}_{n_1}\\
&\dots&\\
x^{(l)}_1+\dots+x^{(l)}_{n_l}&=&y^{(l)}_1+\dots+y^{(l)}_{n_l}.
\end{eqnarray*}
Summing all these solutions up we obtain a solution to the equation
$$x_1+\dots+x_{n}=y_1+\dots+y_{n}$$
in distinct integers  $x_i,y_i\in B$, which contradicts  that $B$ is a $B^*_n$--set.

It remains to prove (\ref{in:X}). Observe that if $l\in \G$ it follows from $2\X\cap \G=\emptyset$ that
$$|\X\cap [l]|\le \big\lceil\frac{l}{2}\big\rceil-1\le \big\lfloor\frac{l-1}{2}\big\rfloor\,,$$ otherwise there are $n_1,n_2\in \X$ such that $n_1+n_2=l,$ which is impossible. Now, let $l\in \X$
and suppose that $|\X\cap [l]|> \lfloor\frac{l}{2}\rfloor$. Let $n$ be the smallest element of $\G$ bigger than $l$. Then
$$\big\lceil\frac{n}{2}\big\rceil-1\ge |\X\cap [n]|=|\X\cap [n-1]|> \big\lfloor\frac{n-1}{2}\big\rfloor\,,$$
which again is a contradiction. This completes the proof.
\hfill$\blacksquare$
\end{proof}

\bigskip

     Clearly, the property $l \mathcal{X} \cap \G = \emptyset$ means that if  $x_1,\dots,x_l\in \mathcal{X}$ such that $x_1+\dots+x_l \in [k]$ then automatically $x_1+\dots+x_l \in \mathcal{X}$.
     Potential  examples of "bad"  $\mathcal{X}$ is the set of even numbers from $[k]$ or $\X=(k/2,k-1]$. However, we do not know any example of a
      dense $B^*_k$--subset of $[N]$ that does not contain a large subset that is $B^*_l$--set for every $l\le k.$

In the proof of the main result we also apply  a variant  of a well--known theorem of Erd{\H o}s--Ko--Rado \cite{ER}, \cite{EKR}.
Let $1\le k,n$, $k< n$ be integer parameters and let $L = \{ l_1< l_2 < \dots < l_r \} \subseteq [0,k]$.
A family $\mathcal{A}$ of subsets of $[n]$ is called $(n,k,L)$--system if,  each of a subset of the family has cardinality $k$ and, for all $A, B\in \mathcal{A}$ from the family we have $|A\cap B|\in L.$ The following theorem was proved in \cite{DEF}.

\begin{theorem}
    Let $\mathcal{A}$ be a $(n,k,L)$--system.
    Then
$$
    |\mathcal{A}| \le \prod_{i=1}^{|L|} \frac{n-l_i}{k-l_i} \,.
$$
\label{t:EKR}
\end{theorem}

Now we are in position  to prove Theorem \ref{t:main_introduction}.
It is a consequence of the following result.

\begin{theorem}
    Let $A$ be a  $B^*_k$--set, $|A| \ge 20 k^{6}$.
    Then
\begin{equation}\label{f:M_k_main}
    M_k (A) \le 2k 8^kk^{3k/2} |A|^k \,.
\end{equation}
\label{t:M_k_main}
\end{theorem}

To see how the theorem above implies Theorem \ref{t:main_introduction} just apply (\ref{f:M_k_main}) and note that by the Cauchy--Schwarz inequality one has
$$
    |A|^{2k} = \Big( \sum_n \sigma_k (n) \Big)^2 \le |kA|\,\sum_n \sigma^2_k (n) 
        \le
           2 k 8^k k^{3k/2} |A|^k k N\,,
$$
hence
$$
    |A| \le 16 k^{3/2} N^{1/k}
$$
as required.

\bigskip

\begin{proof}
     Let $B \subseteq A$ be a set given by  Lemma \ref{l:structure_X}. By the Minkowski inequality we have
\begin{eqnarray*}
    M^{1/2k}_k (A)&=&\big(\int_0^1|f(t)|^{2k}dt\big)^{1/2k}\le \big(\int_0^1|f_B(t)|^{2k}dt\big)^{1/2k}+\big(\int_0^1|f_{A\setminus B}(t)|^{2k}dt\big)^{1/2k}\\
    &\le& M^{1/2k}_k (B) + M^{1/2k}_k (A\setminus B)\,.
\end{eqnarray*}
    Since  $|A\setminus B| \le 2 k^{3}, $ it follows that $M^{1/2k}_k (A\setminus B)\le 2k^3$.
    Furthermore, by
    %(\ref{inequality-Mk})
    the first inequality in (\ref{e:ruzsa-S_k}), we get
    %(more precisely see bound \ref{inequality-Mk}), we get
    $$ M^{1/2k}_k (B)\ge \frac12 (20k^6-2k^3)^{1/2}\ge 2k^3 \ge M^{1/2k}_k (A\setminus B)\,,$$
%$$
%    M^{1/2k}_k (B) \ge (2^{-k} k! |B|^k)^{1/2k} \ge (2^{-k-1} k! |A|^k)^{1/2k} \ge k \cdot 2k^3
%        \ge k M^{1/2k}_k (A\setminus B) \,,
%$$
     so that
$$
    M_{k} (A) \le 2^k M_k (B) \,.
$$
    %With some abuse of the notation we write $A$ for $A_1$.
    Thus, by Lemma \ref{l:M_k=S_k} it is enough to estimate the quantity $S_k (B).$
    %and moreover by Lemma \ref{l:S_k=S_k(A,A^c)}
    %we want
   % it is sufficiently
   % to find an  upper bound for
    %the
   % $S_k (B',B\setminus B')$, where $B'\subseteq B$ is a random set given by Lemma \ref{l:S_k=S_k(A,A^c)}.

       Put $t=\lfloor k/2\rfloor$ and observe that (the formula below works for any $t\le k$)
\begin{equation}\label{S_k-S*_l}
S_t=S_t(B)=\sum_{l=2}^{t}\binom{t}{l}^2(t-l)!S^*_l(B)|B|^{t-l}+t!|B|^t.
\end{equation}
Indeed, notice that from  every solution to the equation
\begin{equation}\label{eq-side-distinct}
 x_1+\dots+x_t = y_1 + \dots+y_t
\end{equation}
in $B$ with $x_i\not=x_j$ and $y_i\not=y_j$, after cancelation of all equal variables from both sides, we obtain for some $l<t,$ a
solution to the equation
\begin{equation}\label{eq-distinct-l}
x_{i_1}+\dots+x_{i_l} = y_{i_1} + \dots+y_{i_l}
\end{equation}
with distinct integers $x_{i_j},y_{i_j}\in B$.
  Conversely, each solution to the equation (\ref{eq-side-distinct})  must be constructed from a solution to (\ref{eq-distinct-l}) for some $2\le l\le t$, and by adding to both sides of the equation the same elements from $B$. This can be done on $\binom{t}{l}^2(t-l)!$ ways, as for a fixed solution $(x_{i_1},\dots,x_{i_l} , y_{i_1}, \dots,y_{i_l})$ we have to chose positions for $x_{i_j}, y_{i_j}$ on $\binom{t}{l}^2$ ways, and then to chose distinct $t-l$ elements from $B$ and order them on the right hand side of the equation on $(t-l)!$ ways. The term $t!|B|^t$ counts the number of solutions with $\{ x_1,\dots,x_t\}=\{ y_1, \dots,y_t\}.$

 Next, applying  Lemma \ref{l:S_k=S_k(A,A^c)} we bound $S^*_l(B)$. Let $B'$ be a set given by Lemma \ref{l:S_k=S_k(A,A^c)}. To estimate $S^*_l(B',B\setminus B')$ observe that every solution to  (\ref{eq-distinct-l}), say,    $(x_{1},\dots,x_{l} , y_{1}, \dots,y_{l}), x_i\in B', y_i\in B\setminus B'$
    corresponds to a set $Z = \{ x_1, \dots, x_l, y_1, \dots, y_l\},$
    and clearly, each such  set with $\sum x_i=\sum y_i,$ gives  $(l!)^2$ solutions to (\ref{eq-distinct-l}).
    Thus, to bound $S^*_l(B', B\setminus B')$ it is sufficient to bound the number of such sets $Z$.
    Denote by $\mathcal{A}_l$ the family of all  sets $Z$ that corresponds to a solution to (\ref{eq-distinct-l}).

    Put $\mathcal{X} = \mathcal{X} (B)$, $\mathcal{G} = \mathcal{G} (B)$.
 %   Property (\ref{f:structure_X}) gives us for any $l$
%\begin{equation}\label{f:X_property_proof}
%    l \mathcal{X} \cap \mathcal{G} = \emptyset \,.
%\end{equation}
    By (\ref{in:X}) we have  $|\mathcal{X} \cap [l]| \le \lfloor\frac l2\rfloor$ for every $1\le l\le k.$
    %assumption $k\in \mathcal{G}$ and hence taking $l=2$ in (\ref{f:X_property_proof}), we obtain
    %$k\notin \mathcal{X}+\mathcal{X}$. In particular, $|\mathcal{X}| \le k/2$.
   % Putting $t=[k/2]$ we see that $2t \le k$ and $|\mathcal{X} \cap [2t]| \le t$.
   % Below we will consider just the intersection $\mathcal{X}$ with the segment $[2t]$ and estimate
   % the quantity $S_l (B',B\setminus B')$. Theorem  \ref{t:EKR}.
Suppose that we have two solutions to (\ref{eq-distinct-l})
    $$
        x_1+\dots+x_a + y_1+\dots+y_{l-a} = x'_1+\dots+x'_b + y'_1+\dots+y'_{l-b}
    $$
    $$
        \tilde{x}_1+\dots+\tilde{x}_a + \tilde{y}_1+\dots+\tilde{y}_{l-a} = \tilde{x}'_1+\dots+\tilde{x}'_b + \tilde{y}'_1+\dots+\tilde{y}'_{l-b} \,,
    $$
    where $0\le a \le l$, $0\le b \le l$, further, $\{x_1,\dots,x_a \} = \{\tilde{x}_1,\dots,\tilde{x}_a \}$,
    $\{x'_1,\dots,x'_b \} = \{\tilde{x}'_1,\dots,\tilde{x}'_b \}$
    and $\{ y_1,\dots,y_{l-a}\} \cap \{ \tilde{y}_1,\dots,\tilde{y}_{l-a}\} = \emptyset$,
    $\{ y'_1,\dots,y'_{l-b}\} \cap \{ \tilde{y}'_1,\dots,\tilde{y}'_{l-b}\} = \emptyset$.
    The above  solutions correspond to some sets $Z_1$, $Z_2 \subseteq [B]^{2l}$ and
    we have $|Z_1 \cap Z_2| = a+b$.
    %It gives us
    Subtracting our solutions, we get
    \begin{equation}\label{tmp:10.06.2016_1}
        y_1+ \dots +y_{l-a} + \tilde{y}'_1+\dots+\tilde{y}'_{l-b}
            =
                y'_1+\dots+y'_{l-b} +  \tilde{y}_1+\dots+\tilde{y}_{l-a} \,.
    \end{equation}
    Because of  $\{ y_1,\dots,y_{l-a}\} \cap \{ \tilde{y}_1,\dots,\tilde{y}_{l-a}\} = \emptyset$,
    $\{ y'_1,\dots,y'_{l-b}\} \cap \{ \tilde{y}'_1,\dots,\tilde{y}'_{l-b}\} = \emptyset$
    we to see that (\ref{tmp:10.06.2016_1}) is a  solution  with distinct elements from $B$.
       Therefore, $2l-|Z_1 \cap Z_2| \in \mathcal{X}$, so that $|Z_1 \cap Z_2| \in (2l - \mathcal{X})$.
    Putting $L= 2l - \mathcal{X} = \{ l_1 < \dots< l_r\} \subseteq [0,2l]$, $r=|\mathcal{X} \cap [2l]| \le l$,  we see that the family $\mathcal{A}_l$ is a $(|B|,2l,L)$--system.
    Using Theorem \ref{t:EKR}, we obtain
$$
    |\mathcal{A}_l| \le \prod_{i=1}^{|L|} \frac{|B|-l_i}{2l-l_i} \le \frac{|B|^r}{r!}
        \le
            \frac{|B|^{l}}{l!} \,.
$$
    Thus,
$$
    S^*_l (B',B\setminus B') \le \frac{|B|^{l}}{l!} \cdot (l!)^2 = l! |B|^l \,,
$$
    so that by Lemma \ref{l:S_k=S_k(A,A^c)}
$$
    S^*_l (B) \le 4^l l! |B|^l \,.
$$
Therefore by (\ref{S_k-S*_l})
$$S_t(B)\le  4^t t!|B|^t\sum_{l=0}^t\binom{t}{l} =  8^t  t! |B|^t\,.$$
    Applying Lemma \ref{l:S_k,S_l} with $l=t$,  we get
\begin{eqnarray*}
    S_k (B) &\le& (\sqrt{2} k)^{2(k-t)} |B|^{k-t}S_t(B)\\
        &\le&
            8^t  t! (\sqrt{2} k)^{2(k-t)} |B|^{k}\\
                                &\le&
                     2^{k+2t} k^{2k-t} |B|^k\\
                   & \le&
                         k4^{k} k^{3k/2} |B|^k \,,
\end{eqnarray*}
hence by
$$M_k(A)\le 2^k M_k(B)\le 2^{k+1} S_k(B) \le 2  k 8^{k} k^{3k/2} |B|^k \le 2  k 8^{k} k^{3k/2} |A|^k\,,$$
which completes the proof.
\hfill$\blacksquare$
\end{proof}

%\bigskip

\bigskip

\noindent{T.~Schoen\\
\no{Faculty of Mathematics and Computer Science,\\ Adam Mickiewicz
University,\\ Umul\-towska 87, 61-614 Pozna\'n, Poland\\} {\tt
schoen@amu.edu.pl}

\bigskip

\noindent{I.D.~Shkredov\\
Steklov Mathematical Institute,\\
ul. Gubkina, 8, Moscow, Russia, 119991}
\\
and
\\
IITP RAS,  \\
Bolshoy Karetny per. 19, Moscow, Russia, 127994
\\
and
\\
MIPT, \\
Institutskii per. 9, Dolgoprudnii, Russia, 141701\\
{\tt ilya.shkredov@gmail.com}

\end{document}